\documentclass[11pt, a4paper]{article}
\usepackage[utf8]{inputenc}
\usepackage{bbm}

\usepackage[margin = 1in]{geometry}
\usepackage{amsfonts,amsmath,amssymb,amsthm}
\usepackage{mathtools}
\usepackage{mathrsfs}
\usepackage{graphicx,subcaption}
\usepackage{booktabs}
\usepackage{tabularx}
\usepackage{lipsum}
\usepackage{hyperref}
\usepackage{multicol}
\hypersetup{
   colorlinks = true,
   citecolor = blue,
   urlcolor = purple,
   linkcolor = red
}

\usepackage{mathrsfs}
\usepackage{multirow}
\usepackage[makeroom]{cancel}
\usepackage{tikz}
\usepackage{framed,color}
\definecolor{shadecolor}{rgb}{1,0.8,0.3}
\usepackage{fancybox}
\usepackage{caption}
\usepackage{algorithm,algorithmic} 
\usepackage{multicol}
\usepackage{aliascnt}
\usepackage{setspace}
\usepackage{longtable}
\allowdisplaybreaks

\usetikzlibrary{decorations.markings,arrows}

\allowdisplaybreaks

\title{\textbf{An Asymptotically Exact Multiple Testing Procedure under Dependence}}

\date{}
\vspace{5mm}
\usepackage{authblk}
\author[1]{Swarnadeep Datta\footnote{\href{swarnadeepdatta0122@gmail.com}{swarnadeepdatta0122@gmail.com}}}
\author[2]{Monitirtha Dey\footnote{\href{mdey@uni-bremen.de}{mdey@uni-bremen.de}}}
\affil[1]{\footnotesize Interdisciplinary Statistical Research Unit, Indian Statistical Institute, Kolkata, India}
\affil[2]{\small Institute for Statistics, University of Bremen, Bremen, Germany}
 
\usepackage{natbib}
\setcitestyle{authoryear, open={(},close={)}}

\begin{document}

\maketitle
\theoremstyle{plain}
\newtheorem{axiom}{Axiom}
\newtheorem{remark}{Remark}
\newtheorem{corollary}{Corollary}[section]
\newtheorem{claim}[axiom]{Claim}
\newtheorem{theorem}{Theorem}[section]
\newtheorem{lemma}{Lemma}[section]
\newtheorem{test}{Test Procedure}
\newtheorem{exa}{Example}
\newtheorem{rem}{Remark}
\newtheorem{proposition}{Proposition}

\newaliascnt{lemmaa}{theorem}
\newtheorem{lemmaa}[lemmaa]{Theorem}
\aliascntresetthe{lemmaa}
\providecommand*{\lemmaautorefname}{Lemma}
\providecommand*{\corollaryautorefname}{Corollary}
\providecommand*{\testautorefname}{Test Procedure}
\providecommand*{\theoremautorefname}{Theorem}
\providecommand*{\propositionautorefname}{Proposition}
\providecommand*{\exaautorefname}{Example}
\providecommand*{\remautorefname}{Remark}

\theoremstyle{definition}
\newtheorem{definition}{Definition}



\theoremstyle{definition}

\begin{abstract}

We propose a simple single-step multiple testing procedure that asymptotically controls the family-wise error rate (FWER) at the desired level exactly under the equicorrelated multivariate Gaussian setup. The method is shown to be asymptotically exact using an explicit plug-in estimator for the equicorrelation, and does not require stepwise adjustments. We establish its theoretical properties, including the convergence to the desired error level, and demonstrate its effectiveness through simulation results. We also spell out related extensions to block-correlated structures and generalized FWER control.
\end{abstract}

\section{Introduction}

The classical theory of statistical inference is derived from the assumption of independent and identically distributed (i.i.d.) data. However, recent technological advances permit modern science to collect, store and explore massive datasets to answer pertinent domain-specific questions. These datasets often showcase inherent dependencies among the observations, necessitating the refinement and tailoring of classical inference tools. These adjustments frequently warrant new proof techniques. Therefore, the topic of statistical inference under dependence has been appealing to statisticians for decades for its novel and intriguing theory and practical relevance. Even more additional challenges arise when one considers the problem of simultaneous inference: controlling overall type I error and maintaining a reasonable power while tackling dependence among test statistics is known to be a formidable task \citep{Dickhaus}.

The equicorrelation structure depicts the simplest form of dependence. Although simple in form, this setting arises in several applications, e.g., when comparing a control against many treatments. Consequently, simultaneous testing of normal means under equicorrelated frameworks has witnessed considerable attention in recent years \citep{delattre, deybhandari_seq, FDR2007, FDR2009, Proschan, royspl}. The limiting behavior of Bonferroni FWER under the equicorrelated Normal setting was considered in \cite{das_2021}. They explicate that the Bonferroni FWER is asymptotically (as the number of hypotheses approaches infinity) a convex function in equicorrelation $\rho$ and hence bounded above by $\alpha(1-\rho)$, $\alpha$ being the target level. Improving their result, \cite{deybhandari} elucidate that the Bonferroni FWER asymptotically goes to zero for any $\rho>0$. \cite{deystpa} shows that this limiting zero phenomenon holds for the general class of step-down FWER-controlling procedures. In a related but different direction, \cite{FDR2007} study the (limiting) empirical distribution function of the p-values and utilize those to study limiting behaviors of FDR of the linear step-up procedure. However, the present literature still lacks a multiple testing procedure that is asymptotically exact under equicorrelation, i.e., which has limiting FWER exactly equal to $\alpha$ under equicorrelation. 

This work aims to close this gap by proposing a simple, single-step, common-cutoff multiple testing procedure. Our method is shown to be asymptotically exact using an explicit plug-in estimator for the equicorrelation, and does not require stepwise adjustments. Related extensions to block correlated structures and generalized FWER control are also spelled out.

The rest of the article is laid out as follows. In Section \ref{sec:prelims} we state our testing framework and revisit existing results on Bonferroni procedure. Section \ref{Sec:our_proposal} presents our proposed procedure and discusses its theoretical properties. We investigate its power properties in Section \ref{sec:power}. Extensions to block correlated structures and to generalized FWER control are explicated in Sections \ref{sec:block} and \ref{sec:kFWER}, respectively. Numerical experiments evaluating the empirical performance of our proposed procedure are presented in Section \ref{sec:sim_study}. We conclude in  Section \ref{sec:discussion}. Proofs are deferred to the appendix. 

Throughout this work, $\phi$ and $\Phi$ respectively denote the p.d.f and c.d.f of $N(0)$ distribution. $\Phi^n(\cdot)$ denotes $[\Phi(\cdot)]^n$. Also, $\mathcal{I}$ represents the set $\{1, \dots, n\}$.

\section{Preliminaries}\label{sec:prelims}

This work considers the equicorrelated Gaussian setup. We have $n$ observations
$$X_i \sim N(\mu_i),   i \in \mathcal{I} $$
where $\operatorname{Corr}(X_i, X_j) = \rho, \ \forall \ i \neq j   \text{ with }  \rho \in (0).$
 We work with unit variances since the literature on the multiple testing theory often assumes that the variances are known. We wish to test:
$$\mathrm{H_{0i}} : \mu_i = 0 \quad \text{vs} \quad \mathrm{H_{1i}} : \mu_i > 0, \quad i \in \mathcal{I}.$$
The global null $\mathrm{H_0} = \bigcap_{i=1}^{n} \mathrm{H_{0i}}$ asserts that each $\mu_i$ is zero.

Let $V_n(T)$ denote the number of false rejections of a multiple testing procedure (MTP henceforth) $T$. The family-wise error rate (FWER) of this procedure is defined as
$$\mathrm{FWER}^T = \mathbb{P}(V_n(T) \geq 1).\label{1st}$$


Under the mentioned framework, if we want to develop a single-step MTP, it is natural to use a right-sided common cutoff for each of $X_1, X_2, \ldots, X_n$, say, $c_n(\alpha, \rho)$. Then the FWER is given as:
$$\mathrm{FWER}= \mathbb{P} \left( \bigcup_{i\in\mathcal{I}_0} \{ X_i > c_n(\alpha, \rho) \} \right),$$
where $\mathcal{I}_0$ is the index set of true nulls. Hence, the FWER under global null $H_0$ is:
$$\mathrm{FWER}_{\mathrm{H_0}}= \mathbb{P}_{\mathrm{H_0}} \left( \bigcup_{i=1}^n \{ X_i > c_n(\alpha, \rho) \} \right).$$

Throughout this work, $n_0:=|\mathcal{I}_0|$ and $n_1=n-n_0$. \cite{das_2021} elucidate that for the equicorrelated Normal setup, the following representation of $\mathrm{FWER}_{H_{0}}$ holds for any single-step testing procedure with common right-sided cutoff $\tau_n(\alpha, \rho)$: 
\begin{equation}
    \mathrm{FWER}_{\mathrm{H_0}}(n, \alpha, \rho)=1-\mathbb{E}_Z \left[ \Phi^n\left( \frac{\tau_n(\alpha, \rho) + \sqrt{\rho} Z}{\sqrt{1 - \rho}} \right) \right], \text{where} Z\sim N(0). \label{4th}
\end{equation}

The classical Bonferroni procedure yields the right-sided cutoff $c_{Bon}(n, \alpha)=\Phi^{-1}(1-\frac{\alpha}{n})$ for any value of equicorrelation $\rho$. Suppose $\mathrm{FWER}^{Bon}(n, \alpha, \rho)$ denotes the Bonferroni FWER under this setting. \cite{deybhandari} showed the following:
\begin{theorem}\label{rev1}
    Given any $\alpha \in (0, 1)$ and $\rho \in (0, 1]$, $\displaystyle \lim_{n \to \infty} \mathrm{FWER}_{\mathrm{H_0}}^{Bon}(n, \alpha, \rho) = 0$.
\end{theorem}

\cite{deybhandari} also extended \autoref{rev1} to general correlation matrices under any configuration of true and false null hypotheses as follows:

\begin{theorem}\label{rev4}
    Suppose $\displaystyle\liminf \rho_{ij} = \delta > 0$. Then, for any $\alpha \in (0, 1)$,
\[
\lim_{n \to \infty} \mathrm{FWER}^{Bon}(n, \alpha, \Sigma_n) = 0
\]
under any configuration of true and false null hypotheses.
\end{theorem}

\section{Proposed Procedure}\label{Sec:our_proposal}

This section consists of two scenarios: when the common correlation $\rho$ is known, and when it is unknown. 

\subsection{Known $\rho$}

Here we suggest the following testing procedure:

\begin{test}\label{test1}
    For given equicorrelation value $\rho\in(0)$ and a target level $\alpha \in (0, 1)$, we define the cutoff
\begin{equation}c_n(\alpha, \rho) :=\sqrt{1-\rho}  \Phi^{-1}\left(1-\frac{1}{n}\right)-\sqrt{\rho}  \Phi^{-1}(\alpha). \label{5th}
\end{equation}
For each $i \in \mathcal{I}$, reject $\mathrm{H_{0i}}$ if $X_i>c_{n}(\alpha, \rho)$.
\end{test}

Note that $\displaystyle\lim_{n \to \infty} \frac{c_n(\alpha, \rho)}{c_{Bon}(n, \alpha)}=\sqrt{1-\rho}<1$, implying that this procedure potentially rejects more hypotheses than the Bonferroni method. Suppose $\mathrm{FWER_{H_{0}}^{I}}(n, \alpha, \rho)$ denotes the FWER under this setting. The following theorem is the key result of this work:

\begin{theorem}\label{fwer1}
   Consider the equicorrelated Gaussian setting with known common correlation $\rho \in (0, 1)$. Then the proposed multiple testing procedure (\autoref{test1}) is asymptotically exact under the global null w.r.t FWER control, i.e.,
   \[ \displaystyle\lim_{n\to\infty}\mathrm{FWER_{\mathrm{H_0}}^{I}}(n, \alpha, \rho)=\alpha.\]
\end{theorem}

The following result extends \autoref{fwer1} to any configuration of true and false null hypotheses under a limiting condition of the proportion of true null hypotheses.



\begin{theorem}\label{fwer2}
Consider the equicorrelated Gaussian setting with known common correlation $\rho \in (0, 1)$. Then, $$\displaystyle\lim_{n\to\infty}\mathrm{FWER^{I}} (n, \alpha, \rho)=\alpha$$ under any configuration of true and false null hypotheses for which $\displaystyle\lim_{n\to\infty}{n_0}/{n}=p_0$ for some $p_0>0$. 
\end{theorem}

\subsection{Unknown $\rho$}

In our equicorrelated normal setup, often the equicorrelation value $\rho$ is unknown. In such scenarios, it is natural to use an estimator $\hat{\rho}$ of $\rho$ (and then use $c_n(\alpha, \hat{\rho})$ as the cutoff instead of $c_n(\alpha, \rho)$). However, we have to ensure that the desired convergence of FWER still holds. The following result provides sufficient conditions for this convergence.

\begin{theorem}\label{fwer3}
    Let $\hat{\rho}_n$ be an estimator of $\rho$ such that $$|c_n(\alpha, \hat{\rho}_n)-c_n(\alpha, \rho)|\longrightarrow0 \quad \text{almost surely as $n\to\infty$}$$ for any configuration of true and false null hypotheses for which $\displaystyle\lim_{n\to\infty}n_0/n>0$. Then, the following holds:
 $$\displaystyle\lim_{n\to\infty}\mathbb{P} \left( \bigcup_{i\in\mathcal{I}_0} \{ X_i > c_n(\alpha, \hat{\rho}_n) \} \right)=\alpha.$$
 In other words, replacing $c_n(\alpha, \rho)$ by $c_n(\alpha, \hat{\rho}_n)$ in \autoref{test1} still ensures the convergence of FWER to $\alpha$.
\end{theorem}

Towards finding such estimator of $\rho$, let $m=\left\lfloor \frac{n}{2} \right\rfloor$. For $1\leq i \leq m$, we define $Y_i=1-\frac{(X_{2i-1}-X_{2i})^2}{2}$. Define $\hat{\rho}_n^{\star}:=\max\{0, \frac{1}{m}\displaystyle\sum_{i=1}^mY_i\}$.

\begin{lemma}\label{lem1}
    Suppose $\frac{n_1 \cdot \sup_i \mu_i^2}{n} \to 0$ as $n \to \infty$. Then, 
    $\hat{\rho}_n^{\star}\longrightarrow\rho$ almost surely as $n \to\infty$. 
    If, moreover, $\frac{\log n \cdot n_1 \cdot \sup_i \mu_i^2}{n} \to 0$ as $n \to \infty$, then $|c_n(\hat{\rho}_n^{\star}, \alpha)-c_n(\alpha, \rho)|\longrightarrow0$ almost surely as $n\to\infty$.
\end{lemma}
\noindent Using the cutoff $c_n(\alpha, \hat{\rho}_n^{\star})$ in place of $c_n(\alpha, \rho)$ in \autoref{test1} we develop the following:

\begin{test}\label{test2}
     For unknown equicorrelation value $\rho\in(0)$ and a target level $\alpha \in (0, 1)$, we define the quantity
\begin{equation}c_n(\alpha, \hat{\rho}_n^{\star}) :=\sqrt{1-\hat{\rho}_n^{\star}}\Phi^{-1}\left(1-\frac{1}{n}\right)-\sqrt{\hat{\rho}_n^{\star}}  \Phi^{-1}(\alpha). \label{6th}
\end{equation}
For each $i \in \mathcal{I}$, reject $\mathrm{H_{0i}}$ if $X_i>c_{n}(\alpha, \hat{\rho}_n^{\star})$.
\end{test}



\autoref{fwer3} and \autoref{lem1} lead to the following result on the convergence of FWER, which is valid not only under the global null, but also under any configuration of true and false null hypotheses satisfying some limiting conditions on $n_1/n$.

\begin{theorem}\label{fwer4}
Consider the equicorrelated Gaussian setting with unknown common correlation $\rho \in (0, 1)$. Then \autoref{test2} is asymptotically exact, i.e, $$\displaystyle\lim_{n\to\infty}\mathrm{FWER^{II}} (n, \alpha, \rho)=\alpha$$ under any configuration of true and false null hypotheses satisfying $\frac{\log n \cdot n_1 \cdot \sup_i \mu_i^2}{n} \to 0$ as $n \to \infty$.
\end{theorem}

\section{Power Analysis} \label{sec:power}

There exist several notions of Power in the realm of simultaneous statistical inference, e.g., the Disjunctive Power (AnyPwr), the Conjunctive Power (AllPwr), the Average Power (AvgPwr) etc \citep{BretzBook, Dudoit, Grandhi, Ramsey1978}. 
Recent works on limiting multiple testing has focused on the Disjunctive power. 
This is defined as the probability of rejecting at least one false null, i.e., \begin{equation}\mathrm{AnyPwr_T}\coloneqq\mathbb{P}(S_n(T)\geq 1).\end{equation}
\cite{deybhandari} proved the following results regarding the asymptotic power of Bonferroni's  procedure:

\begin{theorem}\label{power_old_1}
    Consider the equicorrelated normal setup with equicorrelation $\rho \in (0, 1)$. 
Let $\sup \mu_i$ be finite. Then, for any $\alpha \in (0, 1)$,
\[
\lim_{n \to \infty} \mathrm{AnyPwr_{Bon}}(n, \alpha, \Sigma_n) = 0.
\]
\end{theorem}

\begin{theorem}\label{power_old_2}
Consider the equicorrelated normal setup with equicorrelation $\rho \in (0, 1)$. Suppose $n_1 \to \infty$ and $n_1/n \to p_1 \in (0, 1]$ as $n \to \infty$. Then, for any $\alpha \in (0, 1)$,
\[
\lim_{n \to \infty} \mathrm{AnyPwr_{Bon}}(n, \alpha, \Sigma_n) = 1, \hspace{1mm} \text{if}  \frac{\sqrt{2 \log n_1}}{\mu_{n_1}} \to 0 \quad \text{as } n_1 \to \infty.
\]
\end{theorem}

 The next two results illustrate the superior performance of our suggested MTP with respect to power. 

\begin{theorem}\label{power_new_1}
Consider the equicorrelated normal setup with equicorrelation $\rho \in (0, 1)$.  Let $\sup \mu_i$ be finite. If $\frac{d_n}{n}\cdot l^{\sqrt{2 \log n}} \to \infty$ as $n \to \infty$ for all $l >1$ then, \[\Phi \left(\Phi^{-1}(\alpha)+\frac{\displaystyle\inf_{i\in\mathcal{I}\setminus \mathcal{I}_0}\mu_i}{\sqrt{\rho}}\right) \leq 
\lim_{n \to \infty} \mathrm{AnyPwr_{I}}(n, \alpha, \Sigma_n) \leq  \left(\Phi^{-1}(\alpha)+\frac{\sup \mu_i}{\sqrt{\rho}}\right) .
\]
\end{theorem}

\begin{theorem}\label{power_new_2}
Consider the equicorrelated normal setup with equicorrelation $\rho \in (0, 1)$. Suppose $n_1\to\infty$ and $n_1/n \to p_1 \in (0, 1]$ as $n \to \infty$. Then, for any $\alpha \in (0, 1)$, 
\[
\lim_{n \to \infty} \mathrm{AnyPwr_{I}}(n, \alpha, \Sigma_n) = 1, \hspace{2mm} if \hspace{2mm} \frac{\sqrt{2 \log n_1}}{\mu_{n_1}} < \frac{1}{\sqrt{1-\rho}} \hspace{2mm} as \hspace{2mm} n_1\to \infty.
\]
\end{theorem}

\autoref{power_new_1} ensures that our proposed MTP has asymptotically strictly positive power even when the Bonferroni method has limiting zero power. \autoref{power_new_2} suggests that the condition for consistency is much weaker for our proposal, compared to that of the Bonferroni method.

\section{Extension to Block-Equicorrelation structures}\label{sec:block}

We have focused on equicorrelation structures so far. However, observations often exhibit more complex dependence structures. For example, the block-equicorrelated covariance structure has been used to explicitly model within-replicate and between-replicate correlations of observations from genome-wide data ~\citep{zhu_2007}. We consider such a scenario in this section. Suppose $(X_{1}, X_{2}, \ldots, X_{n})$ have the following covariance structure:

$$\Sigma_{n}:=\left(\begin{array}{cccc}
M_{k_1}(\rho_{1}) & O_{12} & \cdots & O_{1m}\\
O_{21} & M_{k_2}(\rho_{2}) & \cdots & \cdots\\
\vdots& \vdots&  \vdots&  \vdots \\
O_{m1} & \cdots &\cdots & M_{k_m}(\rho_{m})\end{array}\right).$$
Here $k_{j}$'s are positive integers and $\displaystyle\sum_{j=1}^{m}k_{j}=n$. Also, $M_{r}(\rho)$ denotes the $r \times r$ matrix with diagonal entries equal to $1$, off-diagonal entries equal to $\rho$ and $O_{j_{1}j_{2}}$ denotes the zero matrix of order $k_{j_{1}}\times k_{j_{2}}$. In summary, $\Sigma_{n}$ represents a correlation structure having precisely $m$ blocks with $k_{j}$'s being the respective block sizes. Suppose $(\mathcal{I}_1, \ldots, \mathcal{I}_m)$ denotes the corresponding partition of $\mathcal{I}$. In other words, for $1 \leq i\neq i^{\prime} \leq n$,
    $$\operatorname{Corr}(X_i, X_{i^{\prime}}) = \begin{cases}
        \rho_j & if \quad i, i^{\prime} \in \mathcal{I}_j\\
        0 & if \quad i\in \mathcal{I}_r, i^{\prime}\in \mathcal{I}_s \textit{ and } r\neq s.
    \end{cases}$$
For $1 \leq j \leq m$, let 
$$\mathcal{I}_{j0}:= \{i \in \mathcal{I}_j: \mathrm{H_{0i}} \hspace{2mm} \text{is true}\}$$
and suppose $k_{j0}:=|\mathcal{I}_{j0}|$. Hence, $n_0=\displaystyle\sum_{j=1}^m{k_{j0}}$.

We propose the following single-step procedure for asymptotic FWER control under this scenario:
\begin{test}\label{test3}
    Consider the correlated Gaussian sequence model with covariance matrix $\Sigma_n$. For a target level $\alpha \in (0, 1)$, suppose $\beta = 1 - (1-\alpha)^{1/m}$. For each $j \in \{1, \ldots, m\}$, 
     $$\text{for each $i \in \mathcal{I}_j$, reject $\mathrm{H_{0i}}$ if $X_{i} > c_{k_j}(\rho_j, \beta)$.}$$
\end{test}

\noindent The FWER of this procedure is given as 
$$\mathrm{{FWER}^{III}}(n, \alpha, \Sigma_n) = \mathbb{P}\left(\bigcup_{j=1}^m \bigcup_{i \in \mathcal{I}_{j0}} \{X_i >c_{k_j}(\rho_j, \beta)\}\right).$$

Similar to the extension of \autoref{fwer1} for \autoref{test1}, here also, we work for general configuration of true and false nulls with some reasonable assumptions as in \autoref{fwer2} with respect to this block-equicorrelated correlation structure.

\begin{theorem}\label{block}
     Suppose $\displaystyle\min_{j} k_j \to \infty$ as $n \to \infty$. Then, for any configuration of true and false null hypotheses satisfying $\displaystyle\min_j\lim_{n\to\infty}\frac{k_{j_0}}{k_j}>0$, \autoref{test3} is asymptotically exact, i.e., \[\displaystyle\lim_{n\to\infty}\mathrm{FWER^{III}}(n, \alpha, \Sigma_n)=\alpha.\]
\end{theorem}

    \autoref{block} can be extended to even more general correlation structures, using a famous probability inequality due to \cite{Slepian}. Towards this, suppose $(X_{1}, X_{2}, \ldots, X_{n})$ have the following covariance structure:

$$R_{n}:=\left(\begin{array}{cccc}
M_{k_1}(\rho_{1}) & A_{12} & \cdots & A_{1m}\\
A_{21} & M_{k_2}(\rho_{2}) & \cdots & \cdots\\
\vdots& \vdots&  \vdots&  \vdots \\
A_{m1} & \cdots &\cdots & M_{k_m}(\rho_{m})\end{array}\right)$$
where each $A_{j_1 j_2}$ contains non-negative entries. Slepian's inequality leads to the following: 

\begin{corollary}\label{block2}
     Suppose $\displaystyle\min_{j} k_j \to \infty$ as $n \to \infty$. Then, for any configuration of true and false null hypotheses satisfying $\displaystyle\min_j\lim_{n\to\infty}\frac{k_{j_0}}{k_j}>0$, \autoref{test3} asymptotically controls FWER, i.e., \[\displaystyle\lim_{n\to\infty}\mathrm{FWER^{III}}(n, \alpha, R_n) \leq \alpha.\]
\end{corollary}


\section{Extension to $k$-FWER}\label{sec:kFWER}

In practice, one is often willing to tolerate a few false rejections. Hence, controlling $k$ or more false rejections can potentially improve the ability of a procedure to detect more false null hypotheses \citep{deybhandaristpa}. \cite{GuoRao2010} remark that $k$-FWER can be regarded as a good complement to the FWER and FDR in many applications. 

Suppose we rearrange the random variables $X_1, X_2, \ldots, X_n$ into a nondecreasing sequence
$$
X_{1: n} \leqslant X_{2: n} \leqslant \cdots \leqslant X_{n: n} .
$$
We call $X_{n-r+1: n}$ the $r$-th highest order statistic. For fixed $r$, as $n \to \infty$, we call $X_{n-r+1: n}$ \textit{the $r$-th extreme}.

One observes that the $k$-FWER of procedure \ref{test1} is as follows: 
\begin{equation}\label{k-FWER}
    k\text{-}\mathrm{FWER_{H_0}^{I}}(n, \alpha, \rho) = 1 - \mathbb{P}(X_{n-k+1: n} \leq c_n(\alpha, \rho))
\end{equation}
where $c_n(\alpha, \rho)$ is defined as in \eqref{5th}. Thus, the limiting distribution of the $k$-th extreme is crucial to obtain the limit of $k$-FWER. The following result enunciates that the limiting $k$-FWER of \autoref{test1} is $\alpha$.

\begin{theorem}\label{kfwer1}
   Consider the equicorrelated Gaussian setting with common correlation $\rho \in (0, 1)$. Also suppose that $\displaystyle\lim_{n\to\infty}\frac{n_0}{n}>0$. Then, \autoref{test1} is asymptotically exact w.r.t $k$-FWER control, i.e.,
   \[ \displaystyle\lim_{n\to\infty}k\text{-}\mathrm{FWER^{I}}(n, \alpha, \rho)=\alpha.\]
\end{theorem}

The above result discusses the asymptotic of $k$-FWER for fixed $k\in \mathbb{N}$. However, in many scenarios, one may opt for even more rejections by allowing $k$ to be a function of $n$, i. e., $k=k_n$  \citep{deybhandari}. The corresponding asymptotics of $k_n$-FWER naturally depends on the order of $k_n$. The next result comes up with a order of $\{k_n\}_{n\geq1}$ for which $k_n$-FWER is controlled by \autoref{test1}.

\begin{theorem}\label{kfwer2}
Consider the equicorrelated Gaussian setting with common correlation $\rho \in (0, 1)$. Suppose $\{k_n\}_{n\geq1}$ be a sequence for which $n_0-k_n=o\left(\frac{n}{l^{\sqrt{\log n}}}\right)$ for all $l>1$. Then, \autoref{test1} controls $k_n$-FWER asymptotically at $0$, i. e.,
\[ \displaystyle\lim_{n\to\infty}k_n\text{-}\mathrm{FWER^{I}}(n, \alpha, \rho)=0.\]
\end{theorem}

\section{Simulation Study}\label{sec:sim_study}

\subsection{Known $\rho$}
The simulation scheme that we are following to compute FWER values under $H_0$ is similar to that in \cite{deybhandari}.

As mentioned earlier, under $H_0$, $X_i = \theta + Z_i$ where $\theta \sim N(0, \rho)$, independent of $\{Z_n\}_{n \geq 1}$ and $Z_i \overset{\text{iid}}{\sim} N(0, 1-\rho)$. This gives
\begin{align*}
\mathrm{FWER_{H_0}^I}(n, \alpha, \rho) &= \mathbb{P}_{\mathrm{H_0}} \left( \bigcup_{i=1}^{n} \{Z_i + \theta > c_n(\alpha, \rho) \} \right)\\
&= \mathbb{P}_{\mathrm{H_0}} \left( \max_{1 \leq i \leq n} Z_i > c_n(\alpha, \rho) - \theta \right) \\
&= \mathbb{P}_{\mathrm{H_0}} \left( \max_{1 \leq i \leq n} W_i > \frac{c_n(\alpha, \rho) + \sqrt{\rho} \gamma}{\sqrt{1 - \rho}}\right) \quad(\text{here $W_i = Z_i/\sqrt{1-\rho}$,$\theta = -\sqrt{\rho}\gamma$)} \\
&= \mathbb{P}_{\mathrm{H_0}} \left( U > \frac{c_n(\alpha, \rho) + \sqrt{\rho} \gamma}{\sqrt{1 - \rho}}\right) \quad \text{(say)} \\
&= \mathbb{E}_\gamma \left[ \mathbb{I} \left\{U > C_{\gamma, \rho} \right\} \right]
\end{align*}
where $\mathbb{I}\{A\}$ is the indicator variable of event $A$ and $C_{\gamma, \rho}$ is the quantity it is replacing. Note that $\gamma \sim N(0, 1)$ is independent of $U$ under $H_0$.

The above derivation provides us with an elegant and computationally less expensive simulation scheme for estimating $\mathrm{FWER}$ given $(n, \rho)$ and desired $\alpha$ \citep{deybhandari}. Firstly, we generate 100, 000 independent observations from $N(0, 1)$ (these are the $\gamma$ variables, i.e., the repetitions). Given $\rho$, we compute the cutoff $C_{\gamma_i, \rho}$ for each of the simulated $\gamma_i$'s, $1 \leq i \leq 100000$. Given $n$, for each $\gamma_i$, we generate $1$ independent observation from cdf $\Phi^n$, $U_i$. We note for how many $i$'s, $U_i$ exceeds the cutoff $C_{\gamma_i, \rho}$. An estimate of $\mathrm{FWER^I}(n, \alpha, \rho)$ is obtained accordingly from the 100, 000 repetitions.

Tables \ref{tab1}-\ref{tab4} present the estimates of FWER under the equicorrelated normal setup for some specific values of $(n, \alpha, \rho)$. Evidently, the FWER under global null is decreasing to the target $\alpha$ values for each case as $n$ grows, while for each $n$ the values decrease as $\rho$ increases. Also, the convergence of FWER towards $\alpha$ is much faster for higher values of $\rho$.

\renewcommand{\arraystretch}{1.2} 
\setlength{\tabcolsep}{10pt} 

\begin{table}[H]
\centering
\captionsetup{skip=6.5pt} 
\caption{Estimates of $\mathrm{FWER^{I}}$ for \(\alpha=0.15\)}
\begin{tabularx}{\textwidth}{l|*{4}{>{\centering\arraybackslash}X|}>{\centering\arraybackslash}X}
\hline
\textbf{\(\)} & \(\rho=0.1\) & \(\rho=0.3\) & \(\rho=0.5\) & \(\rho=0.7\) & \(\rho=0.9\) \\
\hline
\(n=10^5\)       & 0.28962 & 0.21452 & 0.18832 & 0.17333 & 0.15934 \\
\(n=10^6\)   & 0.27451 & 0.20666 & 0.18595 & 0.16903 & 0.16220 \\
\(n=10^7\)  & 0.26408 & 0.20615 & 0.17941 & 0.17016 & 0.15893 \\
\(n=10^8\) & 0.25683 & 0.19816 & 0.17827 & 0.16727 & 0.15984 \\
\(n=10^9\) & 0.24858 & 0.19308 & 0.17489 & 0.16758 & 0.15672 \\
\hline
\end{tabularx}\label{tab1}
\end{table}

\vspace{5pt} 

\begin{table}[H]
\centering
\captionsetup{skip=6.5pt} 
\caption{Estimates of $\mathrm{FWER^{I}}$ for \(\alpha=0.10\)}
\begin{tabularx}{\textwidth}{l|*{4}{>{\centering\arraybackslash}X|}>{\centering\arraybackslash}X}
\hline
\textbf{\(\)} & \(\rho=0.1\) & \(\rho=0.3\) & \(\rho=0.5\) & \(\rho=0.7\) & \(\rho=0.9\) \\
\hline
\(n=10^5\) & 0.22970 & 0.15726 & 0.13247 & 0.11740 & 0.10814 \\
\(n=10^6\) & 0.21290 & 0.14909 & 0.12761 & 0.11701 & 0.10554 \\
\(n=10^7\) & 0.20621 & 0.14240 & 0.12532 & 0.11484 & 0.10702 \\
\(n=10^8\) & 0.19780 & 0.14016 & 0.12377 & 0.11231 & 0.10618 \\
\(n=10^9\) & 0.18983 & 0.13928 & 0.12076 & 0.11420 & 0.10436 \\
\hline
\end{tabularx}\label{tab2}
\end{table}

\vspace{5pt} 

\begin{table}[H]
\centering
\captionsetup{skip=6.5pt} 
\caption{Estimates of $\mathrm{FWER^{I}}$ for \(\alpha=0.05\)}
\begin{tabularx}{\textwidth}{l|*{4}{>{\centering\arraybackslash}X|}>{\centering\arraybackslash}X}
\hline
\textbf{\(\)} & \(\rho=0.1\) & \(\rho=0.3\) & \(\rho=0.5\) & \(\rho=0.7\) & \(\rho=0.9\) \\
\hline
\(n=10^5\) & 0.15635 & 0.08987 & 0.07137 & 0.06079 & 0.05550 \\
\(n=10^6\) & 0.14217 & 0.08468 & 0.06880 & 0.06020 & 0.05474 \\
\(n=10^7\) & 0.13468 & 0.08094 & 0.06775 & 0.06006 & 0.05407 \\
\(n=10^8\) & 0.12773 & 0.07784 & 0.06654 & 0.05923 & 0.05272 \\
\(n=10^9\) & 0.11952 & 0.07520 & 0.06518 & 0.05708 & 0.05432 \\
\hline
\end{tabularx}\label{tab3}
\end{table}

\vspace{5pt} 

\begin{table}[H]
\centering
\captionsetup{skip=6.5pt} 
\caption{Estimates of $\mathrm{FWER^{I}}$ for \(\alpha=0.01\)}
\begin{tabularx}{\textwidth}{l|*{4}{>{\centering\arraybackslash}X|}>{\centering\arraybackslash}X}
\hline
\textbf{\(\)} & \(\rho=0.1\) & \(\rho=0.3\) & \(\rho=0.5\) & \(\rho=0.7\) & \(\rho=0.9\) \\
\hline
\(n=10^5\) & 0.06749 & 0.02528 & 0.01717 & 0.01345 & 0.01113 \\
\(n=10^6\) & 0.05937 & 0.02299 & 0.01561 & 0.01375 & 0.01091 \\
\(n=10^7\) & 0.05204 & 0.02091 & 0.01543 & 0.01305 & 0.01060 \\
\(n=10^8\) & 0.04668 & 0.01975 & 0.01460 & 0.01288 & 0.01107 \\
\(n=10^9\) & 0.04325 & 0.01920 & 0.01422 & 0.01178 & 0.01084 \\
\hline
\end{tabularx}\label{tab4}
\end{table}


\subsection{Unknown $\rho$}
We employ a similar simulation scheme as in the preceding section. As mentioned earlier, under $H_0$, $X_i = \theta + Z_i$ where $\theta \sim N(0, \rho)$, independent of $\{Z_n\}_{n \geq 1}$ and $Z_i \overset{\text{iid}}{\sim} N(0, 1-\rho)$. This gives
\begin{align*}
\mathrm{FWER_{H_0}^I}(n, \alpha, \rho) &= \mathbb{P}_{\mathrm{H_0}} \left( \bigcup_{i=1}^{n} \{Z_i + \theta > c_n(\alpha, \hat{\rho}_n^{\star}) \} \right)
\end{align*}
where $\hat{\rho}_n^{\star}$ is the estimator of $\rho$ proposed in Section 3.2. 

Tables \ref{tab5}-\ref{tab8} present the estimates of FWER under the equicorrelated normal setup for some specific values of $(n, \alpha, \rho)$. Evidently, the FWER under global null is decreasing to the target $\alpha$ values for each case as $n$ grows, while for each $n$ the values decrease as $\rho$ increases. Here also, the convergence of FWER towards $\alpha$ is much faster for higher values of $\rho$.
\renewcommand{\arraystretch}{1.2} 
\setlength{\tabcolsep}{10pt} 

\begin{table}[H]
\centering
\captionsetup{skip=6.5pt}
\caption{Estimates of $\mathrm{FWER^{II}}$ for \(\alpha=0.15\)}
\begin{tabularx}{\textwidth}{l|*{4}{>{\centering\arraybackslash}X|}>{\centering\arraybackslash}X}
\hline
\textbf{\(\)} & \(\rho=0.1\) & \(\rho=0.3\) & \(\rho=0.5\) & \(\rho=0.7\) & \(\rho=0.9\) \\
\hline
\(n=5000\)  & 0.3074 & 0.2290 & 0.1916 & 0.1724 & 0.1638 \\
\(n=10000\) & 0.3103 & 0.2212 & 0.1927 & 0.1798 & 0.1635 \\
\(n=15000\) & 0.3037 & 0.2248 & 0.1982 & 0.1781 & 0.1539 \\
\(n=20000\) & 0.3097 & 0.2262 & 0.1857 & 0.1783 & 0.1617 \\
\hline
\end{tabularx}\label{tab5}
\end{table}

\vspace{5pt}

\begin{table}[H]
\centering
\captionsetup{skip=6.5pt}
\caption{Estimates of $\mathrm{FWER^{II}}$ for \(\alpha=0.10\)}
\begin{tabularx}{\textwidth}{l|*{4}{>{\centering\arraybackslash}X|}>{\centering\arraybackslash}X}
\hline
\textbf{\(\)} & \(\rho=0.1\) & \(\rho=0.3\) & \(\rho=0.5\) & \(\rho=0.7\) & \(\rho=0.9\) \\
\hline
\(n=5000\)  & 0.2513 & 0.1600 & 0.1403 & 0.1196 & 0.1117 \\
\(n=10000\) & 0.2499 & 0.1660 & 0.1336 & 0.1203 & 0.1106 \\
\(n=15000\) & 0.2445 & 0.1665 & 0.1352 & 0.1191 & 0.1083 \\
\(n=20000\) & 0.2392 & 0.1563 & 0.1328 & 0.1197 & 0.1124 \\
\hline
\end{tabularx}\label{tab6}
\end{table}

\vspace{5pt}

\begin{table}[H]
\centering
\captionsetup{skip=6.5pt}
\caption{Estimates of $\mathrm{FWER^{II}}$ for \(\alpha=0.05\)}
\begin{tabularx}{\textwidth}{l|*{4}{>{\centering\arraybackslash}X|}>{\centering\arraybackslash}X}
\hline
\textbf{\(\)} & \(\rho=0.1\) & \(\rho=0.3\) & \(\rho=0.5\) & \(\rho=0.7\) & \(\rho=0.9\) \\
\hline
\(n=5000\)  & 0.1845 & 0.0974 & 0.0807 & 0.0607 & 0.0584 \\
\(n=10000\) & 0.1728 & 0.0969 & 0.0769 & 0.0642 & 0.0551 \\
\(n=15000\) & 0.1690 & 0.0935 & 0.0736 & 0.0637 & 0.0541 \\
\(n=20000\) & 0.1678 & 0.0963 & 0.0689 & 0.0641 & 0.0567 \\
\hline
\end{tabularx}\label{tab7}
\end{table}

\vspace{5pt}

\begin{table}[H]
\centering
\captionsetup{skip=6.5pt}
\caption{Estimates of $\mathrm{FWER^{II}}$ for \(\alpha=0.01\)}
\begin{tabularx}{\textwidth}{l|*{4}{>{\centering\arraybackslash}X|}>{\centering\arraybackslash}X}
\hline
\textbf{\(\)} & \(\rho=0.1\) & \(\rho=0.3\) & \(\rho=0.5\) & \(\rho=0.7\) & \(\rho=0.9\) \\
\hline
\(n=5000\)  & 0.0884 & 0.0326 & 0.0186 & 0.0143 & 0.0138 \\
\(n=10000\) & 0.0826 & 0.0280 & 0.0182 & 0.0145 & 0.0144 \\
\(n=15000\) & 0.0790 & 0.0304 & 0.0169 & 0.0124 & 0.0104 \\
\(n=20000\) & 0.0773 & 0.0273 & 0.0184 & 0.0138 & 0.0125 \\
\hline
\end{tabularx}\label{tab8}
\end{table}

\section{Discussion}\label{sec:discussion}
This work proposes a simple single-step multiple testing procedure that asymptotically controls the FWER at the desired level exactly under the equicorrelated multivariate Gaussian setup. The proposed procedure exhibits greater rejections than the classical Bonferroni method. The method is shown to be asymptotically exact using an explicit plug-in estimator for the equicorrelation and also does not require stepwise adjustments. 

There are several possible extensions in different directions. One interesting problem would be to study a similar limiting behavior in more general distributional settings. Another viewpoint would be from a Bayesian angle, where studying asymptotic Bayesian optimality of such procedures (as in \cite{Bogdan}) might be intriguing. Finally, extending our limiting results to more general correlation structures would be a challenging problem.


 \newpage 

\section*{Appendix}\label{sec:appendix}

\subsection*{A. Proof of \autoref{fwer1}}

The cutoff sequence $\left\{c_n(\alpha, \rho)\right\}_{n \geq 1}$ of \autoref{fwer1} has been found utilizing the following two crucial results: 


\begin{lemma}\label{appendixlemma2}
    Define $a_{n}=\Phi^{-1}\left(1-\frac{1}{n}\right) \ \forall \ n \geq 1$. Then,
    $$\displaystyle\lim _{n \rightarrow \infty} \Phi^{n}\left(a_{n}+t\right)=\left\{\begin{array}{ll}0 & \text { for } t<0, \\ e^{-1} & \text { for } t=0, \\ 1 & \text { for } t>0. \end{array} \right.$$
\end{lemma}

\begin{lemma}\label{appendixlemma3}
    For any random variable $X$, any $t>0$ and any constant $k$ such that $\mathbb{P}(X=k)=0$, $\displaystyle\lim _{n \rightarrow \infty} \mathbb{E}\left[\Phi^{n}\left(a_{n}+t(X-k)\right)\right]=\mathbb{P}(X>k)$.
\end{lemma}

\begin{proof}[Proof of \autoref{appendixlemma2}]
    Using the Fisher-Tippett-Gnedenko Theorem, \cite{Galambos1978} showed (see section 2.3.2 therein) that
$$
\lim_{n \to \infty} \Phi^{n}\left(a_{n}+b_{n} x\right) = \exp \left(-e^{-x}\right), \hspace{2mm} \text {where} \hspace{2mm} a_{n}=\Phi^{-1}\left(1-\frac{1}{n}\right), b_{n} \approx \frac{1}{a_{n}}.
$$

\noindent \textbf{Case 1: $t<0$.}

For each $x<0$, there exists $N(x) \in \mathbb{N}$ such that $\forall \ n \geq N(x)$, we have $t<b_{n} x$ and hence $\Phi^{n}\left(a_{n}+t\right)<\Phi^{n}\left(a_{n}+b_{n} x\right)$.
Also, note that $\displaystyle\lim _{x \rightarrow -\infty} \exp \left(-e^{-x}\right)=0$.
Therefore, for each $\varepsilon_1>0$, there exists $x_{\varepsilon_1}<0$ such that $\exp \left(-e^{-x_{\varepsilon_1}}\right)<\varepsilon_1$. 

Now, for each $\varepsilon_2>0$, there exists $N(\varepsilon_2) \in \mathbb{N}$ such that for all $n \geq N(\varepsilon_2)$, 
$$|\Phi^{n}\left(a_{n}+b_{n} x_{\varepsilon_1}\right)-\exp \left(-e^{-x_{\varepsilon_1}}\right)|<\varepsilon_2.$$
Hence, for each $\varepsilon_2>0$, there exists $N(\varepsilon_2) \in \mathbb{N}$ such that for all $n \geq N(\varepsilon_2)$, $\Phi^{n}\left(a_{n}+b_{n} x_{\varepsilon_1}\right)<\varepsilon_1+\varepsilon_2$.

\noindent This implies that for each $\varepsilon_1, \varepsilon_2>0$, we always have that for each $n \geq \max \{N(x_{\varepsilon_1}), N(\varepsilon_2)\}$, $$\Phi^{n}\left(a_{n}+t\right)<\varepsilon_1+ \varepsilon_2.$$
Since $\varepsilon_1$ and $\varepsilon_2$ are arbitrary, we conclude that $\displaystyle\lim _{n \rightarrow \infty} \Phi^{n}\left(a_{n}+t\right)=0$. Note that this is true for any $t<0$.

\noindent \textbf{Case 2: $t>0$.}

For each $x>0$, there exists $N(x) \in \mathbb{N}$ such that $\forall \ n \geq N(x)$, we have $t>b_{n} x$ and hence $\Phi^{n}\left(a_{n}+t\right)>\Phi^{n}\left(a_{n}+b_{n} x\right)$.

Also, note that $\displaystyle\lim _{x \rightarrow \infty} \exp \left(-e^{-x}\right)=1$.
Therefore, for each $\varepsilon_1>0$, there exists $x_{\varepsilon_1}>0$ such that $1>\exp \left(-e^{-x_{\varepsilon_1}}\right)>1-\varepsilon_1$. 

Now, for each $\varepsilon_2>0$, there exists $N(\varepsilon_2) \in \mathbb{N}$ such that for all $n \geq N(\varepsilon_2)$, 
$$|\Phi^{n}\left(a_{n}+b_{n} x_{\varepsilon_1}\right)-\exp \left(-e^{-x_{\varepsilon_1}}\right)|<\varepsilon_2.$$
Hence, for each $\varepsilon_2>0$, there exists $N(\varepsilon_2) \in \mathbb{N}$ such that for all $n \geq N(\varepsilon_2)$, $\Phi^{n}\left(a_{n}+b_{n} x_{\varepsilon_1}\right)>1-\varepsilon_1-\varepsilon_2$.

\noindent This implies that for each $\varepsilon_1, \varepsilon_2>0$, we always have that for each $n \geq \max \{N(x_{\varepsilon_1}), N(\varepsilon_2)\}$, $$\Phi^{n}\left(a_{n}+t\right)>1-\varepsilon_1- \varepsilon_2.$$
Since $\varepsilon_1$ and $\varepsilon_2$ are arbitrary, we conclude that $\displaystyle\lim _{n \rightarrow \infty} \Phi^{n}\left(a_{n}+t\right)=1$. Note that this is true for any $t>0$.

\noindent \textbf{Case 3: $t=0$.}

\noindent Here $\Phi^{n}\left(a_{n}+t\right)=\Phi^{n}\left(a_{n}\right)=\Phi^{n}\left(\Phi^{-1}\left(1-\frac{1}{n}\right)\right)=\left(1-\frac{1}{n}\right)^{n}$. So, $$\displaystyle \lim _{n \rightarrow \infty} \Phi^{n}\left(a_{n}+t\right)=\displaystyle \lim _{n \rightarrow \infty}\left(1-\frac{1}{n}\right)^{n}=e^{-1}.$$
Hence, \autoref{appendixlemma2} is proved.
\end{proof}

\begin{proof}[Proof of \autoref{appendixlemma3}]

We use \autoref{appendixlemma2} to prove \autoref{appendixlemma3}. We have a r.v $X$ and a constant $k$ such that $\mathbb{P}(X=k)=0$. Define another r.v $Y:=\mathbb{I}(X>k)$. \autoref{appendixlemma2} implies that for any $t>0$, $$\Phi^{n}\left(a_{n}+t(X-k)\right) \overset{a.s}{\longrightarrow} Y.$$
We also note that $\Phi^{n}(\cdot)$ is bounded between $0$ and $1$. Hence, by dominated convergence theorem, we have,
$$
\lim _{n \rightarrow \infty} \mathbb{E}\left[\Phi^{n}\left(a_{n}+t(X-k)\right)\right]=\mathbb{E}[Y]=\mathbb{P}(X>k).
$$
\end{proof}
\noindent We recall now the expression of the FWER of our proposed procedure: 
$$\mathrm{FWER_{H_{0}}^{I}}\left(n, \alpha, \rho\right)=1-\mathbb{E}\left[\Phi^{n}\left(\frac{c_n(\alpha, \rho)+\sqrt{\rho} Z}{\sqrt{1-\rho}}\right)\right], \hspace{2mm} (Z \sim N(0, 1)).$$ 
Our strategy is to write $1-\mathrm{FWER_{H_{0}}^{I}}\left(n, \alpha, \rho\right)$ as $\mathbb{E}\left[\Phi^{n}\left(a_{n}+t(X-k)\right)\right]$ (for suitably chosen $t$, $X$ and $k$) so that we can utilize \autoref{appendixlemma3}. Towards this, we observe 
\begin{align*}
    & 1 - \lim_{n \to \infty} \mathrm{FWER_{H_{0}}^{I}}\left(n, \alpha, \rho\right)\\
    =  &\lim_{n \to \infty} \mathbb{E}\left[\Phi^{n}\left(\frac{c_n(\alpha, \rho)+\sqrt{\rho} Z}{\sqrt{1-\rho}}\right)\right] \\
    = & \lim _{n \rightarrow \infty} \mathbb{E}\left[\Phi^{n}\left(a_{n}+\frac{\sqrt{\rho}}{\sqrt{1-\rho}}\left(Z-\Phi^{-1}(\alpha)\right)\right)\right] \hspace{2mm}
 \text{(for our chosen $c_n(\alpha, \rho)$)} \\
 = &\mathbb{P}\left(Z>\Phi^{-1}(\alpha)\right) \hspace{2mm}
 \text{(from \autoref{appendixlemma3}, putting $t=\frac{\sqrt{\rho}}{\sqrt{1-\rho}}, k=\Phi^{-1}(\alpha)$)} \\
 = & 1-\alpha.
 \end{align*}
 The rest is obvious.

\subsection*{B. Proof of \autoref{fwer2}}

Without loss of generality, we assume $X_{i} \sim N(0, 1)$ for $1<i \leq n_{0}$ and for $n_{0}+1 \leq i \leq n$, $X_{i} \sim N\left(\mu_{i}, 1\right), \mu_{i}>0$. We are considering ${n_{0}}/{n} \rightarrow p_{0}$ as $n \rightarrow \infty$ for some $p_{0}>0$. Hence, we can consider $\left\{n_{0}\right\}$ as a sequence in $n$.

\noindent \textbf{Claim.} Let $u_{n}>0$ and $v_{n}>0$ be sequences such that $u_{n} \rightarrow a \geq 0$ and $v_{n} \rightarrow b>0$. Then, $u_{n}^{v_{n}} \rightarrow a^{b}$.

\begin{proof}[Proof of Claim]
We consider two cases: 
    
\noindent Case 1: $a>0$. Since $u_{n} \rightarrow a>0$ and $v_{n} \rightarrow b>0$, and the function $f(x, y)=x^{y}$ is continuous on $(0, \infty) \times(0, \infty)$, it follows directly from continuity that $\displaystyle\lim _{n \rightarrow \infty} u_{n}^{v_{n}}=a^{b}$.

\noindent Case 2: $a=0$. Then $u_{n} \rightarrow 0^{+}$and $v_{n} \rightarrow b>0$. Taking logarithms, we write: $u_{n}^{v_{n}}=e^{v_{n} \log u_{n}}$. As
$u_{n} \rightarrow 0^{+}$, we have $\log u_{n} \rightarrow-\infty$, and hence $v_{n} \log u_{n} \rightarrow-\infty$. Thus, $u_{n}^{v_{n}} \rightarrow e^{-\infty}=0=a^{b}$.
\end{proof}
\noindent We take $u_{n}=\Phi^{n}\left(a_{n}+t\right)$ and $v_{n}=\frac{n_{0}}{n}$. \autoref{appendixlemma2}, combined with the preceding claim, give:
$$
\lim _{n \rightarrow \infty} \Phi^{n_{0}}\left(a_{n}+t\right)= \begin{cases}0 & \text { for } t<0 \\ e^{-1} & \text { for } t=0 \\ 1 & \text { for } t>0.\end{cases}
$$
Proceeding similarly, as in the proof of \autoref{fwer1}, completes the rest.

\subsection*{C. Proof of \autoref{lem1}}
\noindent One observes that 
$$X_i \overset{d}{=} \theta + Z_i + \mu_i$$
where $\theta \sim N(0, \rho)$, $Z_i \overset{i.i.d}{\sim} N(0-\rho)$; and $\theta$ and $Z_{i}$'s are independent. Moreover, $\mu_i$ is zero for $i \in \mathcal{I}_0$ and strictly positive otherwise. Now,
\begin{align*}
    Y_i & = 1 - \frac{1}{2}\cdot (X_{2i-1}-X_{2i})^2 \\
        & = 1 - \frac{1}{2}\cdot\left[(\mu_{2i-1}-\mu_{2i}) + (Z_{2i-1}-Z_{2i})  \right]^2
\end{align*}
Hence, 
\begin{align*}
    \frac{1}{m}\displaystyle\sum_{i=1}^m Y_i &  = \frac{1}{m}\displaystyle\sum_{i=1}^m\Bigg[ 1 - \frac{(Z_{2i-1}-Z_{2i}) ^2}{2} \Bigg] - \frac{1}{m} \displaystyle\sum_{i=1}^m (\mu_{2i-1}-\mu_{2i})(Z_{2i-1}-Z_{2i}) - \frac{1}{2m} \displaystyle\sum_{i=1}^m (\mu_{2i-1}-\mu_i)^2\\
    & =\frac{1}{m}\displaystyle\sum_{i=1}^m U_i -  \frac{1}{m} \displaystyle\sum_{i=1}^m V_i -\frac{1}{2m} \displaystyle\sum_{i=1}^m (\mu_{2i-1}-\mu_i)^2 \quad \text{(say).}
\end{align*}
For each $i$, the random variable $U_i$ has mean $\rho$ and variance $2(1-\rho)^2$. Using SLLN, we obtain that $\frac{1}{m}\displaystyle\sum_{i=1}^m U_i \to \rho$ almost surely as $n \to \infty$.

\noindent For each $i$, $V_i$ has mean zero and variance $2(1-\rho)(\mu_{2i-1}-\mu_i)^2$. This implies that, 
\begin{align*}
    \displaystyle\sum_{i=1}^m \frac{\operatorname{Var}(V_i)}{m^2}  = 2(1-\rho) \frac{\displaystyle\sum_{i=1}^m (\mu_{2i-1}-\mu_i)^2}{m^2} &< 2(1-\rho) \frac{\displaystyle\sum_{i \in \mathcal{I} \setminus \mathcal{I}_0} \mu_{i}^2}{m^2}  \\
    & < \infty \quad \text{(from the given condition)}.
\end{align*}
Using Theorem 6.7 of \cite{Petrov}, we obtain that $\frac{1}{m}\displaystyle\sum_{i=1}^m V_i \to 0$ almost surely as $n \to \infty$.

Along similar lines, it is also evident that, under the given condition, $\frac{1}{2m} \displaystyle\sum_{i=1}^m (\mu_{2i-1}-\mu_i)^2 \to 0$ as $n \to \infty$.

Combining each of the limiting behaviors, one obtains $\frac{1}{m} \displaystyle\sum_{i=1}^m Y_i \to \rho$ almost surely as $n \to \infty$. Since $\rho >0$, this implies $\max\{\frac{1}{m} \displaystyle\sum_{i=1}^m Y_i, 0\} \to \rho$ almost surely as $n \to \infty$, completing the proof of the first part. 

\noindent For the second part, consider $r_n: \sqrt{2 \log n}$ and $s_n = [n/2]/r_n^2$. Then, $$\displaystyle\sum_{i=1}^m \frac{\operatorname{Var}(U_i)}{s_n^2} = 2(1-\rho)^2 \cdot \frac{(2 \log n)^2}{\lfloor \frac{n}{2} \rfloor} \longrightarrow 0 \hspace{2mm} \text{as $n \to \infty$}.$$
This implies 
\begin{equation}
    r_n^2 \cdot \left( \frac{1}{m}\displaystyle\sum_{i=1}^m U_i- \rho \right) \to 0 \hspace{2mm} \text{almost surely as $n \to \infty$}. \label{Ubar_convergence}
\end{equation}
We also have, 
$$\displaystyle\sum_{i=1}^m \frac{\operatorname{Var}(V_i)}{s_n^2} =2(1-\rho) \cdot \frac{\displaystyle\sum_{i=1}^m (\mu_{2i-1}-\mu_i)^2}{s_n^2}\to 0 \hspace{2mm} \text{as $n \to \infty$}.$$
Hence, 
\begin{equation}
r_n^2 \cdot \frac{1}{m}\displaystyle\sum_{i=1}^m V_i \to 0 \hspace{2mm} \text{almost surely as $n \to \infty$}.
\label{Vbar_convergence}
\end{equation}
Equations \eqref{Ubar_convergence} and \eqref{Vbar_convergence} suggest that 
$$r_n ^ 2 \cdot \left (\frac{1}{m} \displaystyle\sum_{i=1}^m Y_i -\rho \right) \to 0\hspace{2mm} \text{almost surely as $n \to \infty$}$$
if $r_n ^ 2 \cdot \frac{1}{2m} \displaystyle\sum_{i=1}^m (\mu_{2i-1}-\mu_i)^2 \to 0$ as $n \to \infty$. However $r_n ^ 2 \cdot \frac{1}{2m} \displaystyle\sum_{i=1}^m (\mu_{2i-1}-\mu_i)^2 \to 0$ as $n \to \infty$ if $\frac{\log n \cdot n_1 \cdot \sup_i \mu_i^2}{n} \to 0$ as $n \to \infty$. Hence, under the given conditions, one has 
$$r_n^2 \cdot(\hat{\rho}_n^{\star}-\rho) \to 0\hspace{2mm} \text{almost surely as $n \to \infty$}.$$
The rest is obvious since $a_n \leq r_n$.

\subsection*{D. Proof of \autoref{fwer3}} 
We have 
$$
M_n-c_n(\alpha-\epsilon, \rho)<M_n-\hat{c}_n(\alpha, \rho)<M_n-c_n( \alpha+\epsilon, \rho) \quad \text{almost surely as $n \to \infty$}.
$$
This implies 
$$\left\{M_n-c_n(\alpha+\epsilon, \rho) \leq 0\right\} \subseteq\left\{M_n-\hat{c}_n(\alpha, \rho) \leq 0\right\} \subseteq\left\{M_n-c_n(\alpha-\epsilon, \rho) \leq 0\right\} \quad \text{a.s}.$$
Now, 
\begin{align*}
    \limsup _{n \to \infty} \mathbb{P}\left(M_n-\hat{c}_n(\alpha, \rho) \leq 0\right)  \leq \limsup _{n \to \infty} \mathbb{P}\left(M_n-c_n(\alpha-\epsilon, \rho) \leq 0\right)
    & = 1 - \alpha + \epsilon.
\end{align*}
\begin{align*}
    \liminf _{n \to \infty} \mathbb{P}\left(M_n-\hat{c}_n(\alpha, \rho) \leq 0\right)  \geq \liminf _{n \to \infty} \mathbb{P}\left(M_n-c_n(\alpha+\epsilon, \rho) \leq 0\right)
    & = 1 - \alpha - \epsilon.
\end{align*}
Therefore, 
$$ \lim _{n \to \infty} \mathbb{P}\left(M_n-\hat{c}_n(\alpha, \rho)\leq 0\right)=1-\alpha.$$
The rest follows.

\subsection*{E. Proof of \autoref{power_new_1}}
The following two results would be crucial towards proving \autoref{power_new_1}. 

\begin{proposition}\label{a_n_proposition}
Suppose $a_n = \Phi^{-1}(1-1/n)$. Then, for any $t\neq 0$, 
$$ \lim_{n \to \infty} \frac{1}{\sqrt{2\pi}} \cdot \frac{ e^{-\frac{1}{2} a_n^2} \cdot e^{-a_n t}} {\sqrt{2 \log n}} = \lim_{n \to \infty} \frac{e^{-t \sqrt{2 \log n}}}{n}.
$$    
\end{proposition}

\begin{proposition}\label{Phi^dn}
Suppose $\{d_n\}_{n\geq1}$ is a sequence of $n$ which tends to $\infty$ as $n \to \infty$. Then the following hold.  
\begin{enumerate}
    \item If $\frac{d_n}{n}=o\Big(\frac{1}{l^{\sqrt{\log n}}}\Big)$ for all $l>1$ then for any $t\in\mathbb{R}$,
\[\displaystyle\lim_{n\to\infty}\Phi^{d_n}(a_n+t)=1.\]
\item If $\frac{d_n}{n}\cdot l^{\sqrt{2 \log n}} \to \infty$ as $n \to \infty$ for all $l >1$ then,
$$\displaystyle\lim _{n \rightarrow \infty} \Phi^{d_n}\left(a_{n}+t\right)=\left\{\begin{array}{ll}0 & \text { for } t<0, \\ 1 & \text { for } t>0. \end{array}\right.$$
\end{enumerate}
\end{proposition}

\begin{proof}[Proof of \autoref{a_n_proposition}]
    Throughout this proof, $x\sim y$ indicates that $x/y \to 1$ as $n \to \infty$. The well known result $\displaystyle\lim_{x \to \infty} \frac{x(1-\Phi(x))}{\phi(x)}=1$ gives
    \begin{align*}
 &\frac{\phi\left(a_n\right)}{a_n} \sim  \frac{1}{n} \\
 \implies & \log \phi\left(a_n\right)+\log n-\log a_n \rightarrow 0 \\
\implies & \log \left(\frac{1}{\sqrt{2 \pi}} e^{-a_n^2 / 2}\right)+\log n-\log a_n=o(1) \\
\implies & -\frac{a_n^2}{2}+\log n=\log a_n+\log \sqrt{2 \pi}+o(1) \\
\implies & a_n^2  + 2 \log a_n = 2\log n-  \log 2 \pi+o(1).
\end{align*}
One also has $\displaystyle\lim_{n \to \infty} \frac{a_n}{\sqrt{2 \log n}} =1$. This means $2 \log a_n - \log \log n - \log 2 = o(1)$ as $n \to \infty$. This, combined with the earlier derivation results in 
\begin{equation}
    a_n^2  = 2\log n - \log \log n - \log 4\pi + o(1) 
     = 2 \log n \left (1 - \frac{\log \log n + \log 4\pi}{2 \log n} + o(\frac{1}{\log n})\right) \label{a_n_Approx_1}
\end{equation}
Taking square-roots on both sides and applying Taylor series expansion, one obtains
\begin{equation}
    a_n  = \sqrt{2\log n} \left (1 - \frac{\log \log n + \log 4\pi}{4 \log n} + o(\frac{1}{\log n})\right) 
     = \sqrt{2\log n} - \frac{\log \log n + \log 4\pi}{2 \sqrt{2\log n}} + o(\frac{1}{\sqrt{\log n}}) \label{a_n_Approx_2}
\end{equation}
Now, 
\begin{align*}
    \frac{1}{\sqrt{2\pi}} \cdot \frac{ e^{-\frac{1}{2} a_n^2} \cdot e^{-a_n t}} {\sqrt{2 \log n}}
   \sim & \frac{e^{-\frac{1}{2}\left(2 \log n - \log \log n - \log 4 \pi\right)}}{\sqrt{2\pi}\cdot \sqrt{2 \log n}} \cdot e^{-t(a_n - \sqrt{2 \log n})}  \cdot e^{-t \sqrt{2 \log n}} \quad \text{(using \eqref{a_n_Approx_1})}\\
   \sim & \frac{\frac{1}{n} \cdot \sqrt{\log n}\cdot 2\sqrt{\pi}}{\sqrt{2\pi}\cdot \sqrt{2 \log n}} \cdot e^{-t \sqrt{2 \log n}} \quad \text{(since $a_n-\sqrt{2 \log n} \to 0$ from \eqref{a_n_Approx_2})}\\
   \sim & \frac{e^{-t \sqrt{2 \log n}}}{n},
\end{align*}
completing the proof. 
\end{proof}

\begin{proof}[Proof of \autoref{Phi^dn}]
We start with proving the first part. Under the given conditions of $\{d_n\}_{n\geq1}$, if we can show $\displaystyle\lim_{n\to\infty}\Phi^{d_n}(a_n+t)=1,\text{ for any }t<0$, then the convergence for the other cases follow trivially. For $t<0$,
\begin{align*}
-  \log \left[\Phi^{d_n}\left( a_n+t\right)\right] 
= & - d_n \cdot \log \Phi\left( a_n+t\right) \\
\sim & d_n \cdot\left(1-\Phi\left(a_n+t\right)\right) \\
\sim & d_n \cdot \frac{\phi\left(a_n+t\right)}{a_n+t} \\
\sim & d_n \cdot \frac{e^{-t^2/2}}{\sqrt{2\pi} } \cdot \frac{ e^{-\frac{1}{2} a_n^2} \cdot e^{- a_n t}}{\sqrt{2 \log n}} \\
\sim & \frac{e^{-t^2/2}}{\sqrt{2\pi} } \cdot \frac{d_n}{\sqrt{2 \log n}} \cdot e^{- \log n}\\
 = &  e^{-t^2/2} \cdot d_n \cdot \frac{e^{-t \sqrt{2 \log n}}}{n}  \quad \text{(utilizing \autoref{a_n_proposition}).} 
\end{align*}
Thus, for $t<0$, $\displaystyle\lim_{n\to\infty}\log \left[\Phi^{d_n}\left(a_n+t\right)\right] = 0$ when $d_n =o \left( \frac{n}{e^{-t\sqrt{2 \log n}}} \right)$. The first part is established. 

\noindent Consider the second part now. For $t<0$, again we have, 
$$-  \log \left[\Phi^{d_n}\left( a_n+t\right)\right]  \sim  e^{-t^2/2} \cdot d_n \cdot \frac{e^{-t \sqrt{2 \log n}}}{n} .$$
Thus, under the given condition in the second part, for $t<0$, we have
\[\displaystyle\lim_{n\to\infty}\Phi^{d_n}(a_n+t)=0.\]
For $t>0$, one readily obtains $e^{-t\sqrt{2 \log n}} \to 0$ as $n \to \infty$. The rest follows.
\end{proof}
\noindent We prove \autoref{power_new_1} now. One has
$$\begin{aligned}
\mathrm{AnyPower_{I}}(n, \alpha, \rho) & =1 - \mathbb{E}\left[\prod_{i \in \mathcal{I}\setminus \mathcal{I}_0} \Phi\left(\frac{c_n(\alpha, \rho)+\sqrt{\rho} Z-\mu_{{i}}}{\sqrt{1-\rho}}\right)\right], \quad Z \sim N(0, 1) \\
& \geq 1 - \mathbb{E}\left[\prod_{i \in \mathcal{I}\setminus \mathcal{I}_0}^{} \Phi\left(\frac{c_n(\alpha, \rho)+\sqrt{\rho} Z-\inf \mu_i}{\sqrt{1-\rho}}\right)\right]\\
& = 1 - \mathbb{E}\left[ \Phi^{n_1}\left(a_n +\frac{\sqrt{\rho}}{\sqrt{1-\rho}}\cdot \left(Z-\Phi^{-1}(\alpha)-\frac{\inf \mu_i}{\sqrt{\rho}} \right)\right)\right].
\end{aligned}$$
Taking limits on both sides, we obtain 
$$\lim_{n \to \infty} \mathrm{AnyPower_{I}}(n, \alpha, \rho) \geq 1 - \mathbb{P}(Z> \Phi^{-1}(\alpha)+\frac{\displaystyle\inf_{i\in\mathcal{I}\setminus \mathcal{I}_0} \mu_i}{\sqrt{\rho}})$$
using \autoref{Phi^dn} and \autoref{appendixlemma2}. The other inequality is established similarly. 

\subsection*{F. Proof of \autoref{power_new_2}}
We have
$$\begin{aligned}
1 - \mathrm{AnyPower_{I}(n, \alpha, \rho) } & =\mathbb{E}\left[\prod_{i \in \mathcal{I}\setminus \mathcal{I}_0} \Phi\left(\frac{c_n(\alpha, \rho)+\sqrt{\rho} Z-\mu_{{i}}}{\sqrt{1-\rho}}\right)\right], \quad Z \sim N(0, 1) \\
& \leq \mathbb{E} \left[\Phi\left(\frac{\sqrt{1-\rho}\cdot a_n+\sqrt{\rho} Z-\mu_{n^{\star}}}{\sqrt{1-\rho}}\right)\right]
\end{aligned}$$
for some $\mu_{n^\star}$. The rest follows from the given condition since $\displaystyle\lim_{n \to \infty} \frac{a_n}{\sqrt{2 \log n}} =1$.

\subsection*{G. Proof of \autoref{block}}
The definition of the proposed procedure results in
\begin{align*}
    1 - \lim_{n \to \infty} \mathrm{FWER^{III}}\left(n, \alpha, \Sigma_n\right) 
    & = \lim_{n \to \infty} \mathbb{P}\left(\bigcap_{j=1}^m \bigcap_{i \in \mathcal{I}_{j0}} \{X_i \leq c_{k_j}(\rho_j, \beta)\}\right)\\
    & = \lim_{n \to \infty} \prod_{j=1}^{m} \mathbb{P}  \left(\bigcap_{i \in \mathcal{I}_{j0}} \{X_i \leq c_{k_j}(\rho_j, \beta)\}\right) \\
    & =  \prod_{j=1}^{m} \lim_{n \to \infty} \mathbb{P}  \left(\bigcap_{i \in \mathcal{I}_{j0}} \{X_i \leq c_{k_j}(\rho_j, \beta)\}\right) \\
 & = \prod_{j=1}^{m} (1-\beta) \hspace{2mm} \text{(using \autoref{fwer2})} \\
 & = (1-\beta)^{m} \\
 & = 1-\alpha \hspace{2mm} \text{(from the definition of $\beta$).}
 \end{align*}

 \subsection*{H. Proof of \autoref{kfwer1}}

We utilize the following result on the limiting distribution of extreme order statistics:

\begin{lemma}\label{galambos_k} (Theorem 2.8.1 of \cite{Galambos1978}) For real-valued sequences $a_n$ and $b_n>0$, and for a fixed integer $k>1$, as $n \rightarrow\infty$,
$$
F_{n-k+1: n}\left(a_n+b_n x\right)=\mathbb{P}\left(X_{n-k+1: n}<a_n+b_n x\right)
$$
converges weakly to a nondegenerate distribution function $H^{(k)}(x)$ if and only if,
$$
H_n\left(a_n+b_n x\right)=F_{n: n}\left(a_n+b_n x\right)$$ converges weakly to a nondegenerate distribution function $H(x)$. 

\noindent  If $H^{(k)}(x)$ exists, then for $x$ satisfying $0<H(x)<1$,
$$
H^{(k)}(x)=H(x) \cdot \displaystyle\sum_{i=0}^{k-1} \frac{\left[- \log H(x)\right]^i}{i!},
$$
where $H(x)$ is one of the three extreme value distributions.
\end{lemma}

\noindent Using Lemma \ref{galambos_k}, we establish the following generalization of Lemma \ref{appendixlemma2}: 
\begin{proposition}\label{prop1}
Suppose $X_1, \ldots, X_n$ are i.i.d random variables having $N(0, 1)$ distribution. Let $\Phi_{n-k+1:n}$ denote the cdf of $X_{n-k+1:n}$, where $X_{n-k+1:n}$ is the $k$-th order statistic. Define $a_{n}=\Phi^{-1}\left(1-\frac{1}{n}\right) \ \forall \ n \geq 1$. Then, $$\displaystyle\lim _{n \rightarrow \infty} \Phi_{n-k+1:n}\left(a_{n}+t\right)=\left\{\begin{array}{ll}0 & \text { for } t<0, \\ e^{-1}\displaystyle\sum_{i=0}^{k-1} \frac{1}{i!} & \text { for } t=0, \\ 1 & \text { for } t>0. \end{array}\right.$$
\end{proposition}

\begin{proof}[Proof of Proposition \ref{prop1}]
In our setting, $F=\Phi$. As mentioned earlier, we have 
$$
\lim_{n \to \infty} \Phi^{n}\left(a_{n}+b_{n} x\right) = \exp \left(-e^{-x}\right), \hspace{2mm} \text {where} \hspace{2mm} a_{n}=\Phi^{-1}\left(1-\frac{1}{n}\right), b_{n} \approx \frac{1}{a_{n}}.
$$
In other words, for Normal distribution, with these choices of $a_n\text{ and }b_n$, $H(x) = \exp \left(-e^{-x}\right)$. Hence, applying Lemma \ref{galambos_k} for finite $k \in \mathbb{N}$, we obtain
\begin{equation}\label{k_limit}
\lim_{n \to \infty} \Phi_{n-k+1: n}\left(a_n+b_n x\right) = H^{(k)}(x):=H(x) \cdot \displaystyle\sum_{i=0}^{k-1} \frac{\{- \log H(x)\}^i}{i!}
\end{equation}
where $H(x) = \exp \left(-e^{-x}\right)$. A little computation gives
$H^{(k)}(x)  = \displaystyle\sum_{i=0}^{k-1} \frac{e^{-(e^{-x}+ix)}}{i!}$.

The rest of the proof proceeds the exact similar way as in the proof of Lemma \ref{appendixlemma2} by considering the three cases: $t<0$, $t>0$, and $t=0$ and then accordingly as in \autoref{fwer2}. \end{proof}

\noindent One observes that, under the global null
\begin{align*}
    k\text{-}\mathrm{FWER_{H_0}^{I}}(n, \alpha, \rho) & = 1 - \mathbb{P}(X_{n-k+1: n} \leq c_n(\alpha, \rho))\\
    &= 1 - \mathbb{P}(-\sqrt{\rho}\cdot Z+ \sqrt{1-\rho}\cdot Z_{n-k+1: n} \leq c_n(\alpha, \rho)) \\
    &= 1 - \mathbb{P} \left(  Z_{n-k+1: n} \leq \frac{c_n(\alpha, \rho)) + \sqrt{\rho}Z}{\sqrt{1-\rho}}\right) \\
    &= 1 - \mathbb{P} \left(  Z_{n-k+1: n} \leq a_n + \frac{\sqrt{\rho}}{\sqrt{1-\rho}}\left(Z-\Phi^{-1}(\alpha)\right)\right)
\end{align*}

\noindent Let $Y:=\mathbb{I}(Z>\Phi^{-1}(\alpha))$. Then, similar to the proof of \autoref{appendixlemma3}, one obtains 
$$\mathbb{P} \left(  Z_{n-k+1: n} \leq a_n + \frac{\sqrt{\rho}}{\sqrt{1-\rho}}\left(Z-\Phi^{-1}(\alpha)\right)\right) \overset{a.s}{\longrightarrow} Y.$$
Using the Dominated Convergence Theorem, one readily gets 
$$\displaystyle\lim_{ n \to \infty} k\text{-}\mathrm{FWER_{H_0}^{I}}(n, \alpha, \rho) = 1 - \mathbb{E}(Y) = \alpha.$$
\noindent

\subsection*{I. Proof of \autoref{kfwer2}}



\noindent We have 
$$
\begin{aligned}
k_n\text{-}\mathrm{F W E R}(n, \alpha, \rho) & =\mathbb{P}\left(X_{\left(n_0-k_n+1\right)}>c_n(\alpha, \rho)\right) \\
& =1-\mathbb{E}_Z\left[\binom{n_0}{k_n-1} \Phi^{n_0-k_n+1}\left(a_n+\frac{\sqrt{\rho}}{\sqrt{1-\rho}}\left(Z-\Phi^{-1}(\alpha)\right)\right)\right].
\end{aligned}
$$
Now, for each $k_n \leq n_0$ and for each $t \in \mathbb{R}$, one has $\binom{n_0}{k_n-1} \Phi\left(a_n+t\right) > 1$ for all sufficiently large $n$. Thus, one has, for all sufficiently large $n$,
\begin{equation}
k_n\text{-}\mathrm{F W E R}(n, \alpha, \rho) \leq 1-\mathbb{E}_Z\left[\Phi^{n_0-k_n}\left(a_n+\frac{\sqrt{\rho}}{\sqrt{1-\rho}}\left(Z-\Phi^{-1}(\alpha)\right)\right)\right].\label{k_n_fwer}
\end{equation}
Now, putting $d_n=n_0-k_n$ in \autoref{Phi^dn} and then using DCT on the second term in RHS of \ref{k_n_fwer}, the rest follows.

\section*{Statements and Declaration}

The authors hereby state that they do not have any relevant financial or non-financial competing interest.
\vspace{4mm}

\bibliography{references}

\end{document}